\documentclass[a4paper,11pt,reqno]{amsart}

\usepackage[T1]{fontenc}
\usepackage{lmodern}
\usepackage{microtype}
\usepackage[a4paper,left=24mm,right=24mm,top=22mm,bottom=24mm]{geometry}
\usepackage{mathtools,amssymb,amsthm}
\usepackage{xcolor}
\colorlet{leanblue}{blue!55!black}
\usepackage{tikz}
\usepackage{quiver}
\usepackage[shortlabels]{enumitem}
\usepackage{needspace}
\usepackage{listings}
\lstdefinelanguage{GAP}{
  morekeywords={for,in,do,od,if,then,fi,and,or,not,rec},
  sensitive=true,
  morecomment=[l]{\#},
  morestring=[b]"
}
\lstdefinestyle{gapcode}{
  language=GAP,
  basicstyle=\ttfamily\footnotesize,
  keywordstyle=\bfseries\color{leanblue},
  commentstyle=\itshape\color{black!55},
  stringstyle=\color{black!75},
  numbers=left,
  numberstyle=\tiny\color{black!45},
  numbersep=8pt,
  frame=single,
  rulecolor=\color{black!25},
  framesep=5pt,
  xleftmargin=18pt,
  xrightmargin=5pt,
  columns=fullflexible,
  keepspaces=true,
  showstringspaces=false,
  breaklines=true,
  captionpos=b,
  aboveskip=.5\baselineskip,
  belowskip=.5\baselineskip
}
\usepackage{hyperref}

\newcommand{\Z}{\mathbb{Z}}
\newcommand{\U}{\operatorname{U}}
\newcommand{\Sym}{\operatorname{Sym}}
\newcommand{\gr}{\operatorname{gr}}
\newcommand{\ab}{\operatorname{ab}}

\newcommand{\presentationlangle}{%
  \mathopen{\vcenter{\hbox{%
    \begin{tikzpicture}[x=1ex,y=1ex,baseline=(current bounding box.center)]
      \draw[line width=.45pt] (2,17.5) -- (0,0) -- (2,-17.5);
    \end{tikzpicture}%
  }}}%
}
\newcommand{\presentationrangle}{%
  \mathclose{\vcenter{\hbox{%
    \begin{tikzpicture}[x=1ex,y=1ex,baseline=(current bounding box.center)]
      \draw[line width=.45pt] (0,17.5) -- (2,0) -- (0,-17.5);
    \end{tikzpicture}%
  }}}%
}

\newtheorem{maintheorem}{Theorem}

\newtheorem{maincorollary}[maintheorem]{Corollary}

\numberwithin{equation}{section}
\newtheorem{Theorem}[equation]{Theorem}
\newtheorem{Corollary}[equation]{Corollary}
\newtheorem{Lemma}[equation]{Lemma}
\newtheorem{Proposition}[equation]{Proposition}

\theoremstyle{definition}
\newtheorem{Example}[equation]{Example}

\newcommand{\leanrepository}{%
  https://github.com/TimmPeterson/lie-algebras/blob/main}
\newcommand{\leanproof}[1]{%
  \leavevmode\unskip\nobreak\hfill
  \mbox{\href{#1}{\normalfont\sffamily\bfseries\footnotesize
    \textcolor{leanblue}{Lean 4}}}%
  \par\nobreak\noindent\ignorespaces
}

\title{Dimension subrings of Lie rings}
\author{Vasily Ionin}
\address{Saint Petersburg Department of Steklov Mathematical Institute\\
27 Fontanka, St. Petersburg, 191023, Russia}
\email{ionin.code@gmail.com}
\author{Roman Mikhailov}
\address{Saint Petersburg State University\\
7--9 Universitetskaya Embankment, St. Petersburg, 199034, Russia}
\author{Tima Petrov}
\address{Saint Petersburg State University\\
7--9 Universitetskaya Embankment, St. Petersburg, 199034, Russia}
\email{timapetrov9@gmail.com}
\date{}

\hypersetup{
  colorlinks=true,
  allcolors=blue!55!black,
  pdftitle={Dimension subrings of Lie rings},
  pdfauthor={Vasily Ionin, Roman Mikhailov, and Tima Petrov}
}

\subjclass[2020]{17B35, 17B30, 20C05}
\keywords{Lie ring, dimension subring, lower central series, universal
enveloping algebra, metabelian Lie ring}

\begin{document}

\begin{abstract}
We show that, unlike the lower central series, the dimension series of a Lie
ring need not stabilize when two consecutive terms coincide; in fact,
arbitrarily long finite plateaux occur.  We prove that
$D_n(L)/\gamma_{n+1}(L)$ is central in $L/\gamma_{n+1}(L)$ for every
$n\geq1$, in contrast to the group case.  If $c\geq2$ and $L$ is metabelian and nilpotent of class at
most $c$, we prove that $D_{2c-1}(L)=0$.  In particular,
our results imply that $D_5(L)\subseteq\gamma_4(L)$ for every Lie ring $L$.
\end{abstract}

\maketitle

\section{Introduction}

Let $G$ be a group.  Its lower central series is defined by
\[
 \gamma_1(G)=G,
 \qquad
 \gamma_{n+1}(G)=[\gamma_n(G),G].
\]
Let $\Delta(G)$ denote the augmentation ideal of the integral group ring
$\Z[G]$.  The dimension subgroups of $G$ are
\[
 D_n(G)=G\cap\bigl(1+\Delta(G)^n\bigr)
       =\{g\in G\mid g-1\in\Delta(G)^n\}.
\]
Induction gives
$\gamma_n(G)\subseteq D_n(G)$ for every $n$.  Equality holds for
$n\leq3$, but not in general: Rips constructed a group with
$D_4(G)\neq\gamma_4(G)$~\cite{Rips1972}.  For free groups, equality in all
degrees goes back to Magnus~\cite{Magnus1937}; Witt established the
corresponding result for free Lie rings~\cite{Witt1937}.

Describing dimension subgroups, and understanding how they differ from the
lower central series, is a difficult and fundamental problem in group theory.
The possible discrepancy is remarkably rich: every abelian group of bounded
exponent embeds in a dimension quotient $D_n(G)/\gamma_n(G)$ of a suitable
group.  The construction realizes this torsion through homotopy groups of
spheres~\cite{BartholdiMikhailov2023}.  We refer
to~\cite{Passi1979,MikhailovPassi2009} for the classical dimension-subgroup
problem and its subsequent development.

There is a parallel theory for Lie rings.  For a Lie ring $L$, let $\U(L)$
be its universal enveloping algebra and let
$\varpi(L)$ be the kernel of the
augmentation $\U(L)\to\Z$.  The lower central series and the dimension
series of $L$ are
\[
 \gamma_1(L)=L,
 \qquad
 \gamma_{n+1}(L)=[\gamma_n(L),L],
 \qquad
 D_n(L)=L\cap\varpi(L)^n.
\]
Lie dimension quotients can also contain elements of every prime order:
for each prime $p$, there exist a Lie ring $L$ and an index $n$ for
which $D_n(L)/\gamma_n(L)$ contains a subgroup isomorphic to $\Z/p\Z$.
The topological input has a striking parallel form:
just as homotopy groups of spheres are essential ingredients in constructing
higher torsion in group dimension quotients, the initial pages of unstable
spectral sequences for spheres play the same role for Lie
rings~\cite{BartholdiMikhailov2023}.

For Lie rings over $\Z$, we identify $L$ with its image in $\U(L)$ by
the integral Birkhoff--Witt theorem.  As in the group case,
\[
                     \gamma_n(L)\subseteq D_n(L).
\]
Equality $D_n(L)=\gamma_n(L)$ holds for $n\leq3$, while a Lie ring analogue
of the Rips example shows that it can fail in degree
four~\cite{BartholdiPassi2015}.

The subject also has important homological and categorical aspects.
Dimension quotients are closely connected with derived functors of
non-additive functors~\cite{HartlMikhailovPassi2007}.  More recently, they
have been described as boundary limits over categories of free
presentations, both for groups and for Lie rings~\cite{Mikhailov2026}.

\subsection{Failure of stabilization}

Once it became clear that the dimension series need not coincide with the
lower central series, a broader circle of questions emerged: which structural
features do the two filtrations share, and where do they fundamentally
diverge?  One difference is already visible at the level of universal
algebra.  For each $n$, the groups satisfying $\gamma_n(G)=1$ form a variety
defined by the single identity
\[
                         [x_1,\ldots,x_n]=1.
\]
By contrast, the groups satisfying $D_4(G)=1$ form a quasivariety that is not
finitely based; equivalently, this class cannot be axiomatized by finitely many
quasi-identities~\cite{MikhailovPassi2006}.

We exhibit another fundamental difference.  The lower central series has a
rigid stabilization property: once
$\gamma_m(L)=\gamma_{m+1}(L)$, all subsequent terms equal $\gamma_m(L)$.
The dimension series behaves differently.  An arbitrarily long string of
equal consecutive terms need not mark the point of stabilization.

\Needspace{9\baselineskip}
\begin{maintheorem}\label{thm:A}
\leanproof{\leanrepository/LieRings/Publication/Results.lean\#L32-L42}
For every integer $N\geq1$, there exist an integer $m=m(N)\geq1$ and a Lie
ring $L_N$ such that
\[
 D_m(L_N)=D_{m+1}(L_N)=\cdots=D_{m+N}(L_N)
 \neq D_{m+N+1}(L_N).
\]
\end{maintheorem}

An explicit construction of $L_N$ is given in
\S~\ref{sec:plateaux}.  In this construction, $m(N)=N+4$, and $L_N$
is metabelian and nilpotent of class $N+3$.

In \S~\ref{subsec:group-plateau}, we also construct a group $G$ satisfying
$D_5(G)=\cdots=D_{5+N}(G)\neq D_{5+N+1}(G)$ for some integer $N\geq1$.

\subsection{Centrality of dimension subrings}

The commutator identities also differ in the two settings.
For every group $G$ and every integer $n\geq1$, the quotient
$D_n(G)/\gamma_{n+1}(G)$ is abelian~\cite{GuptaKuzmin1992}; equivalently,
\[
                    [D_n(G),D_n(G)]\subseteq\gamma_{n+1}(G).
\]
The corresponding abelianity statement holds for every Lie ring
$L$~\cite{BartholdiPassi2015}.

Since $D_3(G)=\gamma_3(G)$, one has
\[
                         [D_3(G),G]=\gamma_4(G).
\]
In degree four, one has the further identity
\[
                         [D_4(G),G]=\gamma_5(G)
\]
for every group $G$~\cite{GuptaKuzmin1992}.  This pattern does not continue:
in general, $D_n(G)/\gamma_{n+1}(G)$ need not lie in the center of
$G/\gamma_{n+1}(G)$~\cite{GuptaKuzmin1995}.  For metabelian Lie rings,
the identity
\[
                         [D_n(L),L]=\gamma_{n+1}(L)
\]
holds in every degree~\cite{PassiSicking2017}.  The following theorem removes
the metabelian hypothesis.

\begin{maintheorem}\label{thm:B}
\leanproof{\leanrepository/LieRings/Publication/Results.lean\#L44-L55}
Let $L$ be a Lie ring.  For all integers $m,n\geq1$,
\[
                 [D_n(L),D_m(L)]\subseteq\gamma_{n+m}(L).
\]
In particular,
\([D_n(L),L]=\gamma_{n+1}(L)\).
\end{maintheorem}

\subsection{Plotkin's problems for metabelian Lie rings}

In 1973, Plotkin formulated two problems that remain
open~\cite{Plotkin1973}.  For a group $G$, put
\[
 D_\omega(G)=\bigcap_{n\geq1}D_n(G),
 \qquad
 \gamma_\omega(G)=\bigcap_{n\geq1}\gamma_n(G).
\]
The problems ask:
\begin{enumerate}[(i)]
 \item Is $D_\omega(G)=\gamma_\omega(G)$ for every group $G$?
 \item For every positive integer $c$, is there an integer $m=m(c)$ such that
 $D_m(G)=1$ for every nilpotent group $G$ of class at most $c$?
\end{enumerate}
The analogous questions for Lie rings remain open in general.  We
resolve both Lie-ring analogues for metabelian Lie rings.  The key result is
the following uniform vanishing bound.

\begin{maintheorem}\label{thm:C}
\leanproof{\leanrepository/LieRings/Publication/Results.lean\#L57-L65}
Let $c\geq2$ be an integer, and let $L$ be a metabelian Lie ring of nilpotency class at most $c$. Then
\(D_{2c-1}(L)=0\).
\end{maintheorem}

This answers the second of Plotkin's questions in the metabelian Lie setting.
It also shows that the construction in Theorem~\ref{thm:A} is sharp with
respect to the nilpotency class: $L_N$ has class $N+3$ and
$D_{2N+4}(L_N)\neq0$, whereas every metabelian Lie ring $L$ of class at
most $N+3$ satisfies \(D_{2N+5}(L)=0\).

\begin{maincorollary}\label{cor:D}
\leanproof{\leanrepository/LieRings/Publication/Results.lean\#L67-L77}
Let $c \ge 2$ be an integer, and let $L$ be a Lie ring. Then
\(D_{2c-1}(L)\subseteq\gamma_{c+1}(L)+L''\).
In particular, $D_5(L)\subseteq\gamma_4(L)$.
\end{maincorollary}

The corresponding question for groups remains open: is
$D_5(G)\subseteq\gamma_4(G)$ for every group $G$?
See~\cite[Problem~20.84]{KourovkaNotebook21}.

The first of Plotkin's questions also has an affirmative answer in the
metabelian Lie setting.  For a Lie ring $L$, define
\[
 D_\omega(L)=\bigcap_{n\geq1}D_n(L),
 \qquad
 \gamma_\omega(L)=\bigcap_{n\geq1}\gamma_n(L).
\]

\begin{maincorollary}\label{cor:E}
\leanproof{\leanrepository/LieRings/Publication/Results.lean\#L79-L94}
Let $L$ be a metabelian Lie ring. Then
\[
              D_{2c-1}(L)\subseteq\gamma_{c+1}(L),
              \qquad c\geq2.
\]
In particular, $D_\omega(L)=\gamma_\omega(L)$.
\end{maincorollary}

\subsection{Organization}

In \S~\ref{sec:preliminaries}, we record the integral PBW theorem and the
filtrations used below.  In \S\S~\ref{sec:plateaux} and~\ref{sec:centrality},
we prove Theorems~\ref{thm:A} and~\ref{thm:B}, respectively.
In \S~\ref{sec:relative-vanishing}, we establish a vanishing criterion for
$D_n(L)$, which we apply in \S~\ref{sec:metabelian_lie_rings} to prove
Theorem~\ref{thm:C}.  In \S~\ref{sec:formalization}, we present the Lean
formalization and the GAP computation.

\medskip
The authors thank L.~Bartholdi for helpful discussions of questions related
to this paper.

\subsection{AI use statement}

We used ChatGPT-6 Astra to simplify proofs, generate the Lean~4 formalization,
and assist with text editing.  The authors are responsible for the final
manuscript.

\section{Preliminaries}\label{sec:preliminaries}

\subsection{Notation}

A Lie ring is a Lie algebra over~\(\Z\).
All commutators are left-normed:
\[
 [x_1,x_2,\ldots,x_r]=[[\cdots[[x_1,x_2],x_3],\ldots],x_r].
\]
Let $L$ be a Lie ring. For additive subgroups $B_1,\ldots,B_r$ of $L$, the notation
$[B_1,\ldots,B_r]$ denotes the subgroup generated by the corresponding
commutators.  For an ideal $A$ of $L$, set
\[
 [A,{}_0L]=A,
 \qquad
 [A,{}_{r+1}L]=[[A,{}_rL],L].
\]
We write $L'=[L,L]$ and $L''=[L',L']$, and denote the center of $L$ by
$Z(L)$. The Lie ring $L$ is metabelian if $L''=0$.

All tensor products are taken over $\Z$. Let $L$ be a Lie ring. Denote by
$T(L)=\bigoplus_{n\geq 0}L^{\otimes n}$ the tensor algebra of the additive
group of $L$. The universal enveloping algebra of $L$ is the quotient
\[
\U(L)=T(L)\big/\left( x\otimes y-y\otimes x-[x,y]
  \mathrel{\big|} x,y\in L\right).
\]
Let $\iota\colon L\to \U(L)$ denote the canonical additive homomorphism.
The augmentation homomorphism
$\varepsilon\colon \U(L)\to\Z$ is induced by the projection of
$T(L)$ onto $L^{\otimes 0}=\Z$. We write
\(
\varpi(L)=\ker(\varepsilon)
\)
for the augmentation ideal of $\U(L)$.

\subsection{Dimension subrings}

The lower central series of $L$ is defined inductively by
\[
\gamma_1(L)=L,
\qquad
\gamma_{n+1}(L)=[\gamma_n(L),L],\qquad n\geq 1.
\]
The dimension series of $L$ is defined by
\[D_n(L)=\iota^{-1}\left(\varpi(L)^n\right), \quad n \ge 1.\]

For every $n\geq1$,
\begin{equation}\label{eq:lower-in-dimension}
                         \gamma_n(L)\subseteq D_n(L).
\end{equation}
Indeed, $\iota(L)\subseteq\varpi(L)$ and
$[\varpi(L)^r,\varpi(L)]\subseteq\varpi(L)^{r+1}$, so
\eqref{eq:lower-in-dimension} follows by induction.  Both series are
functorial: a homomorphism $f\colon L\to M$ satisfies
\[
 f(\gamma_n(L))\subseteq\gamma_n(M),
 \qquad
 f(D_n(L))\subseteq D_n(M).
\]

We also use the following result of Passi and Sicking.
\begin{Theorem}[{\cite[Theorem~1.1]{PassiSicking2017}}]\label{thm:Passi-Sicking}
Let $L$ be a metabelian Lie ring. Then
\[
             2 \cdot D_n(L)\subseteq\gamma_n(L), \quad n \ge 1.
\]
\end{Theorem}

\subsection{The Poincar\'e--Birkhoff--Witt theorem}

The filtration of $T(L)$ by tensor length induces an increasing
filtration
\[
0=P_{\leq -1}\U(L)\subseteq P_{\leq 0}\U(L)
\subseteq P_{\leq 1}\U(L)\subseteq P_{\leq 2}\U(L)\subseteq\cdots
\]
on $\U(L)$, where $P_{\leq r}\U(L)$ is the additive subgroup of $\U(L)$
generated by products of at most $r$ elements of $L$. In particular,
\[
P_{\leq 0}\U(L)=\Z\cdot 1,
\qquad
P_{\leq 1}\U(L)=\langle\{1\}\cup\iota(L)\rangle.
\]

Consider the associated graded algebra
\[
\gr\U(L)=\bigoplus_{r\geq 0}\gr_r\U(L),
\qquad
\gr_r\U(L)
=\frac{P_{\leq r}\U(L)}{P_{\leq r-1}\U(L)},\qquad r\geq 0.
\]
It is a commutative connected $\Z_{\geq 0}$-graded algebra.

\begin{Theorem}[Poincar\'e--Birkhoff--Witt]
\label{thm:pbw}
Let $L$ be a Lie ring. Then the
canonical additive homomorphism $\iota\colon L\to \U(L)$ is injective, and
the canonical homomorphism of commutative rings
\[
\iota_*\colon \Sym_{\Z}(L)
\longrightarrow \gr\U(L)
\]
is an isomorphism of graded algebras.
\end{Theorem}
This is the Poincar\'e--Birkhoff--Witt theorem over the Dedekind domain $\Z$; see
Higgins~\cite{Higgins1969}.
Henceforth we identify every Lie ring with its image in its universal
enveloping algebra.

If the additive group of $L$ is a direct sum of cyclic groups, as is the
case, for example, when the additive group of $L$ is finitely generated, Theorem~\ref{thm:pbw} admits the
following explicit form.

\begin{Corollary}
\label{cor:pbw-basis}
Let $L$ be a Lie ring, and let $\mathcal X\subseteq L$ be a totally
ordered subset such that
\(
L=\bigoplus_{x\in\mathcal X}\langle x\rangle
\)
as an abelian group. Set $M_0=\Z\cdot 1\subseteq \U(L)$. For
$r\geq 1$, let $M_r\subseteq \U(L)$ be the additive subgroup
\[M_r = \langle x_1\cdots x_r \mid x_i\in\mathcal X, \; x_1\leq x_2\leq\cdots\leq x_r\rangle.\]
Then
\[\U(L)=\bigoplus_{r\geq 0}M_r,
\qquad
P_{\leq r}\U(L)=\bigoplus_{s=0}^r M_s,
\quad r\geq 0.
\]
In particular,
\(
\gr_r\U(L)\cong M_r
\) for each \(r \geq 0\).
Moreover, for $r\geq1$,
\[
M_r=\bigoplus_{\substack{x_1,\ldots,x_r\in\mathcal X\\
x_1\leq\cdots\leq x_r}}\langle x_1\cdots x_r\rangle,
\]
and each generator $x_1\cdots x_r$ has additive order
\[
\operatorname{ord}(x_1\cdots x_r)
=\gcd\left(\operatorname{ord}(x_1),\ldots,\operatorname{ord}(x_r)\right).
\]
Here we use the convention that $\operatorname{ord}(x)=0$ when $x$ has
infinite order.
\end{Corollary}
\begin{proof}
By Theorem~\ref{thm:pbw},
$P_{\leq r}\U(L)/P_{\leq r-1}\U(L)\cong\Sym^r_{\Z}(L)$, whose ordered monomials
generate a direct sum of cyclic groups with the stated orders.
Each finite order of a factor annihilates the corresponding product in
$\U(L)$; B\'ezout's identity therefore shows that their greatest common
divisor also annihilates it when positive.  Thus sending each monomial
class to the same product in $\U(L)$ defines a section of
$P_{\leq r}\U(L)\to P_{\leq r}\U(L)/P_{\leq r-1}\U(L)$ with image $M_r$.
This proves the assertion about $M_r$ and gives
$P_{\leq r}\U(L)=P_{\leq r-1}\U(L)\oplus M_r$.
Induction on $r$ yields the displayed decompositions.
\end{proof}

\section{Finite plateaux}\label{sec:plateaux}

\begin{flushright}
\begin{minipage}{0.62\textwidth}
\itshape
\guillemotleft\,Un plateau est toujours au milieu, ni début ni
fin.\,\guillemotright

\smallskip
\hfill---Gilles Deleuze and Félix Guattari
\end{minipage}
\end{flushright}

Fix an integer $N\geq1$.  We construct a metabelian Lie ring $L_N$ and
an element $a$ of order two, and prove
\[
 D_{N+4}(L_N)=\langle a\rangle,
 \qquad
 a\in D_{2N+4}(L_N)\setminus D_{2N+5}(L_N).
\]
Since the dimension series is descending, these assertions determine the
plateau and the first zero term.

\subsection{Construction of the Lie rings}

Let $F_{\mathrm{met}}$ be the free metabelian Lie ring on $x_1,\ldots,x_5$.
Order the generators by
\[
                         x_5<x_1<x_2<x_3<x_4.
\]
The standard metabelian Hall basis consists of the generators and the
commutators
\[
 [x_{i_0},x_{i_1},\ldots,x_{i_r}],
 \qquad r\geq1,\qquad
 x_{i_0}>x_{i_1}\leq x_{i_2}\leq\cdots\leq x_{i_r}.
\]
Put
\[
 t=[x_4,\underbrace{x_5,\ldots,x_5}_{N}],
 \qquad
 c_i=[t,x_i],
 \qquad
 u_i=[c_i,x_i],
 \quad 1\leq i\leq3.
\]
Thus $t$, $c_i$, and $u_i$ have weights $N+1$, $N+2$, and $N+3$,
respectively.  With the chosen order, the three $u_i$ are distinct Hall
basis elements.

Define the relators
\begin{equation}\label{eq:shifted-relators}
 \begin{aligned}
 r_1&=4x_1+2c_3+c_2,\\
 r_2&=16x_2+4c_3-c_1,\\
 r_3&=64x_3-4c_2-2c_1.
 \end{aligned}
\end{equation}
Define $L_N$ by imposing the following relations on $F_{\mathrm{met}}$:
\begin{enumerate}[(i)]
 \item $r_1=r_2=r_3=0$;
 \item $\gamma_{N+4}(F_{\mathrm{met}})=0$;
 \item all Hall basis elements of weight $N+3$ other than
 $u_1,u_2,u_3$ are zero;
 \item
 \[
                    u_1=16u_3,
                    \qquad u_2=4u_3,
                    \qquad 64u_3=0.
 \]
\end{enumerate}

We use the same symbols for elements of $F_{\mathrm{met}}$ and their images in $L_N$.

\Needspace{19\baselineskip}
\begin{Example}\label{ex:first-plateau}
Let $\mathcal H_q$ denote the weight-$q$ part of the metabelian Hall basis
fixed above.  A subscript $\mathrm{met}$ indicates a presentation in the
variety of metabelian Lie rings.  Substituting $N=1$ gives
\[
\begin{aligned}
L_1={}&
\presentationlangle
\left.
x_1,x_2,x_3,x_4,x_5
\ \middle|\
\begin{aligned}
\text{(i)}\;&\quad
 \begin{gathered}
 4x_1+2[x_4,x_5,x_3]+[x_4,x_5,x_2]=0,\\
 16x_2+4[x_4,x_5,x_3]-[x_4,x_5,x_1]=0,\\
 64x_3-4[x_4,x_5,x_2]-2[x_4,x_5,x_1]=0;
 \end{gathered}\\[5pt]
\text{(ii)}\;&\quad \gamma_5=0;\\[5pt]
\text{(iii)}\;&\quad
 \begin{gathered}
 h=0\quad\text{for }h\in\mathcal H_4\setminus
 \{[x_4,x_5,x_1,x_1],\\
 \hfill [x_4,x_5,x_2,x_2],[x_4,x_5,x_3,x_3]\};
 \end{gathered}\\[5pt]
\text{(iv)}\;&\quad
 \begin{gathered}
 [x_4,x_5,x_1,x_1]=16[x_4,x_5,x_3,x_3],\\
 [x_4,x_5,x_2,x_2]=4[x_4,x_5,x_3,x_3],\\
 64[x_4,x_5,x_3,x_3]=0.
\end{gathered}
\end{aligned}
\right.
\presentationrangle_{\mathrm{met}}.
\end{aligned}
\]
\end{Example}

\begin{Lemma}\label{lem:top-layer}
The Lie ring $L_N$ is metabelian and nilpotent of class $N+3$, and
\[
             \gamma_{N+3}(L_N)=\langle u_3\rangle\cong\Z/64\Z.
\]
\end{Lemma}

\begin{proof}
Metabelianity is inherited from $F_{\mathrm{met}}$, and relation~\textup{(ii)} gives
$\gamma_{N+4}(L_N)=0$.  By relation~\textup{(iii)},
$\gamma_{N+3}(L_N)$ is generated by $u_1,u_2,u_3$.  Relation~\textup{(iv)}
reduces it to $\langle u_3\rangle$, with $64u_3=0$.
We will show that $u_3$ has order $64$.  Define an additive homomorphism
$\mu\colon F_{\mathrm{met}}/\gamma_{N+4}(F_{\mathrm{met}})\to\Z/256\Z$ by
\[
 \begin{aligned}
 \mu([x_2,x_1])&=4,&\mu([x_3,x_1])&=2,&
 \mu([x_3,x_2])&=1,\\
 \mu(u_1)&=64,&\mu(u_2)&=16,&\mu(u_3)&=4,
 \end{aligned}
\]
and by zero on all other Hall basis elements.  We claim that $\mu$
annihilates the relations \textup{(i)--(iv)}, and hence descends to $L_N$.

The relators in \textup{(iii)} and \textup{(iv)} are central modulo
$\gamma_{N+4}(F_{\mathrm{met}})$ and have zero image under $\mu$.
It remains to check the $r_i$ and their iterated brackets with the generators.
Write $r_i=e_ix_i+C_i$, where $(e_1,e_2,e_3)=(4,16,64)$ and
$C_i\in\gamma_{N+2}(F_{\mathrm{met}})$, as in~\eqref{eq:shifted-relators}.
Substitution into the definition of $\mu$ gives $\mu(r_i)=0$ and
$\mu([r_i,x_j])=0$.

For $s\geq2$, the term $[C_i,x_{j_1},\ldots,x_{j_s}]$ lies in
$\gamma_{N+2+s}(F_{\mathrm{met}})\subseteq\gamma_{N+4}(F_{\mathrm{met}})$.  Hence, modulo
$\gamma_{N+4}(F_{\mathrm{met}})$,
\[
 [r_i,x_{j_1},\ldots,x_{j_s}]
       =e_i[x_i,x_{j_1},\ldots,x_{j_s}].
\]
The commutator on the right has weight $s+1\geq3$, so only $u_1,u_2,u_3$
can contribute to its image under $\mu$.
Expansion in the Hall basis preserves the number of occurrences of each
generator.  Since the commutator contains $x_i$, only $u_i$ among these
three elements can occur.  Thus the image is a multiple of
$e_i\mu(u_i)=256=0$ in $\Z/256\Z$.
This proves that $\mu$ annihilates the defining ideal.  The claim is proved.
Since $\mu(u_3)=4$ has order $64$, so does $u_3$.
\end{proof}

\subsection{Computation of the dimension series}

Set
\[
 b_{12}=[x_1,x_2],
 \qquad
 b_{13}=[x_1,x_3],
 \qquad
 b_{23}=[x_2,x_3]
\]
and define
\begin{equation}\label{eq:a-definition}
                 a=32b_{12}+64b_{13}+128b_{23}\in L_N.
\end{equation}

For distinct $i,j\in\{1,2,3\}$, metabelianity gives
$[c_i,x_j]=[t,x_i,x_j]=[t,x_j,x_i]$, and relation~\textup{(iii)} makes
this mixed commutator zero.  Bracketing
\eqref{eq:shifted-relators} with $x_1,x_2,x_3$ therefore gives
\begin{equation}\label{eq:six-relations}
 \begin{array}{rclrcl}
 4b_{12}+u_2&=&0,& 4b_{13}+2u_3&=&0,\\
 -16b_{12}-u_1&=&0,&16b_{23}+4u_3&=&0,\\
 -64b_{13}-2u_1&=&0,&-64b_{23}-4u_2&=&0.
 \end{array}
\end{equation}

\begin{Lemma}\label{lem:a-order}
In $L_N$ one has $a=32u_3$.  In particular,
\(\langle a\rangle\cong\Z/2\Z\).
\end{Lemma}

\begin{proof}
The relations involving $16b_{12}$, $64b_{13}$, and $64b_{23}$
in~\eqref{eq:six-relations} give
\[
 32b_{12}=-2u_1,
 \qquad
 64b_{13}=-2u_1,
 \qquad
 128b_{23}=-8u_2.
\]
Since $u_1=16u_3$, $u_2=4u_3$, and $64u_3=0$, it follows that
\[
 a=-4u_1-8u_2=-96u_3=32u_3.
\]
By Lemma~\ref{lem:top-layer}, the element $u_3$ has order $64$; hence
$32u_3$ has order $2$.
\end{proof}

\Needspace{6\baselineskip}
\begin{Lemma}\label{lem:a-height}
The element $a$ belongs to $D_{2N+4}(L_N)$.
\end{Lemma}

\begin{proof}
Since $t,c_i\in L_N'$ and $L_N$ is metabelian, $[t,c_i]=0$.
Bracketing~\eqref{eq:shifted-relators} with $t$ gives
\begin{equation}\label{eq:c-orders}
                    4c_1=0,\qquad16c_2=0,\qquad64c_3=0.
\end{equation}
The relations~\eqref{eq:shifted-relators}, together with~\eqref{eq:c-orders}, yield
\[
 \begin{aligned}
 32x_2+64x_3&=4(c_1+c_2-2c_3),\\
 -32x_1+128x_3&=16(c_2+c_3),\\
 -64x_1-128x_2&=64c_3.
 \end{aligned}
\]
Expanding~\eqref{eq:a-definition} in $\U(L_N)$ and substituting these
identities gives
\begin{equation}\label{eq:a-factorization}
 \begin{aligned}
 a={}&(4x_1)(c_1+c_2-2c_3)
      +(16x_2)(c_2+c_3)+(64x_3)c_3\\
   ={}&c_1c_3-c_2^2.
 \end{aligned}
\end{equation}
For the last equality, substitute for $4x_1,16x_2,64x_3$ from
\eqref{eq:shifted-relators} and use the commutativity of $c_1,c_2,c_3$
in $\U(L_N)$.  Since each $c_i$ belongs to
$\gamma_{N+2}(L_N)\subseteq\varpi(L_N)^{N+2}$, the result follows.
\end{proof}

\begin{Lemma}\label{lem:critical-normal-form}
One has $D_{N+4}(L_N)=\langle a\rangle$.
\end{Lemma}

\begin{proof}
Take a free Lie presentation $L_N=E/R$ and put $S=E'+R$.
Since $L_N$ is metabelian, $[E',S]=E''+[E',R]\subseteq R$.
Thus the reduction of Passi and Sicking
\cite[\S~2]{PassiSicking2017}, applied to the preabelian presentation
with generators ordered as $x_1<\cdots<x_5$, gives
\[
 D_n(L_N)\subseteq
 \gamma_n(L_N)+
 \left\langle e_j[x_j,x_i]\mid i<j\right\rangle,
 \qquad
 (e_1,e_2,e_3,e_4,e_5)=(4,16,64,0,0).
\]
Since $\gamma_{N+4}(L_N)=0$,
relations~\eqref{eq:six-relations} give
\[
 D_{N+4}(L_N)\subseteq
 \langle16b_{12},64b_{13},64b_{23}\rangle
 =\langle u_1,2u_1,4u_2\rangle
 =\langle16u_3\rangle.
\]
Theorem~\ref{thm:Passi-Sicking} gives $2D_{N+4}(L_N)=0$.
As $u_3$ has order $64$, it follows that
$D_{N+4}(L_N)\subseteq\langle32u_3\rangle=\langle a\rangle$.
The reverse inclusion follows from Lemma~\ref{lem:a-height}, since
$2N+4\geq N+4$.
\end{proof}

\begin{Lemma}\label{lem:last-initial-form}
The element $a$ does not belong to $\varpi(L_N)^{2N+5}$.
\end{Lemma}

\begin{proof}
We construct an additive homomorphism from $\U(L_N)$ to $\Z/256\Z$ which vanishes on
$\varpi(L_N)^{2N+5}$ and takes the value $128$ on $a$.

First, let $T$ be obtained from $F_{\mathrm{met}}$ by imposing \textup{(ii)--(iv)} and
$[r_i,x_j]=0$ for $1\leq i\leq3$, $1\leq j\leq5$.
The ideal $K=\langle r_1,r_2,r_3\rangle$ is central and $T/K=L_N$.
All defining relations of $T$ lie in $F_{\mathrm{met}}'$, so
\[
 T=T'\oplus\bigoplus_{j=1}^5\Z x_j
 \quad\text{as abelian groups}.
\]
The images of the $r_i$ in $T/T'\cong\Z^5$ are the independent elements
$4x_1,16x_2,64x_3$.  Hence $K\cap T'=0$ and $T'\cong L_N'$.

Put $(e_1,e_2,e_3)=(4,16,64)$.  The Hall-coordinate maps $\rho_i$ which send
$c_i$ to $1\bmod e_i$ and all other Hall commutators to zero descend to $T'$:
relations \textup{(iii)--(iv)} have weight $N+3$, and a bracket of $r_j$
can contribute to the $c_i$-coordinate only by a multiple of $e_i$.
Also, \eqref{eq:c-orders} holds in $T$.  Setting $A=\bigcap_i\ker\rho_i$,
we obtain
\[
 T'=\langle c_1\rangle\oplus\langle c_2\rangle\oplus\langle c_3\rangle
       \oplus A,\qquad
 \operatorname{ord}(c_i)=e_i,\qquad \gamma_{N+3}(T)\subseteq A.
\]
Choose generators for a decomposition of $A$ as a direct sum of cyclic
groups.  Order these and $c_1<c_2<c_3$
before $x_1<\cdots<x_5$.

Use the PBW decomposition from Corollary~\ref{cor:pbw-basis}
to define an additive map $\lambda\colon\U(T)\to\Z/256\Z$.
On $T\subseteq\U(T)$, set
\[
 \lambda|_T\colon T\twoheadrightarrow L_N\xrightarrow{\mu}\Z/256\Z,
\]
where $\mu$ is defined as in the proof of Lemma~\ref{lem:top-layer}.
On ordered monomials of length two, prescribe the following nonzero values:
\[
 \begin{array}{c|rrrrrrr}
 M&x_1^2&x_1x_2&c_2x_1&c_2x_2&c_2^2&c_2c_3&c_3^2\\ \hline
 \lambda(M)&-8&-4&32&16&128&128&192.
 \end{array}
\]
Set $\lambda(1)=0$ and let $\lambda$ vanish on all remaining PBW monomials
of length at least two.
These values respect the additive orders of the monomials, so $\lambda$
is well defined.  In particular,
$\lambda(\gamma_{N+3}(T)\,T)=0$.

We check that $\lambda$ annihilates $K\U(T)$.
Write $r_i=e_ix_i+C_i$, with
\[
 C_1=c_2+2c_3,\qquad C_2=-c_1+4c_3,\qquad C_3=-2c_1-4c_2.
\]
The table and the values of $\mu$ give
$\lambda(r_i)=\lambda(r_ix_j)=\lambda(r_ic_j)=0$.
They also give $\lambda(r_iq)=0$ for every $q\in A$.
For a PBW monomial of length at least two, collect its product with $r_i$
by moving $e_ix_i$ to its ordered position.
Since $r_i$ is central and $T'$ is abelian,
$e_i[x_i,T']=0$ and
$e_i[x_i,x_j]=-[C_i,x_j]\in\gamma_{N+3}(T)$.
Thus collection gives monomials of length at least three and products
of an element of $\gamma_{N+3}(T)$ with a nonempty monomial.
All have zero image under $\lambda$.
Consequently, $\lambda$ descends to $\U(T)/K\U(T)=\U(L_N)$.

Finally, we prove that $\lambda(\varpi(T)^{2N+5})=0$.
The ideal $\varpi(T)^{2N+5}$ is generated additively by words of length
at least $2N+5$ in the $x_i$.
Collect such a word by placing all factors from $T'$ before the ordered
generators $x_i$.  Each resulting product consists of commutators whose
total weight is at least $2N+5$.
Terms with at least three factors have zero image under $\lambda$.
Terms with one factor vanish because $\gamma_{N+4}(T)=0$.
In a term with two factors, each factor has weight at most $N+3$,
so both have weight at least $N+2$ and one has weight $N+3$.
Its image under $\lambda$ is therefore also zero.
Hence $\lambda$ vanishes on $\varpi(L_N)^{2N+5}$, whereas
$\lambda(a)=\mu(32u_3)=128\neq0$.
\end{proof}

\begin{proof}[Proof of Theorem~\ref{thm:A}]
Take $m(N)=N+4$.  For $N+4\leq q\leq2N+4$,
Lemmas~\ref{lem:a-height} and~\ref{lem:critical-normal-form} give
\[
 \langle a\rangle\subseteq D_{2N+4}(L_N)
 \subseteq D_q(L_N)\subseteq D_{N+4}(L_N)=\langle a\rangle.
\]
Thus all these terms equal $\langle a\rangle\cong\Z/2\Z$ by
Lemma~\ref{lem:a-order}.
The next term is a subgroup of $\langle a\rangle$ and does not contain
$a$, by Lemma~\ref{lem:last-initial-form}; hence $D_{2N+5}(L_N)=0$.
\end{proof}

\subsection{A group-theoretic plateau}\label{subsec:group-plateau}

We now construct a group with two consecutive equal nontrivial dimension
subgroups followed by a trivial term, using
\cite[Example~2.16]{MikhailovPassi2009}.  We use the convention
$[x,y]=x^{-1}y^{-1}xy$.  Put
\[
                  d_i=[x_4,x_6,x_i],\qquad i\in\{1,2,3,5\},
\]
and let $\Gamma$ be the group generated by $x_1,\ldots,x_6$ with defining
relators
\begin{equation}\label{eq:group-plateau-presentation}
\begin{aligned}
 &x_1^8d_3^4d_2,\qquad
 x_2^{64}d_3^{-16}d_1^{-1}d_5^{16},\qquad
 x_3^{16384}d_2^{16}d_1^{-4},\qquad
 x_5^{1024}d_2^{16},\\
 &d_5^{2048},\qquad d_1^{16},\qquad d_2^{128},\qquad
 [d_1,x_1][d_2,x_2]^{-8},\qquad
 [d_2,x_2]^{-8}[d_3,x_3]^{2048}.
\end{aligned}
\end{equation}
(Here we have replaced $x_7$ by $x_6$ in the third
relator of the printed presentation, so that all $d_i$ have the form
$[c,x_i]$ for the same element $c=[x_4,x_6]$.
The calculation below uses
this corrected presentation.)
Set
\[
                  M=\Gamma/\bigl(\gamma_5(\Gamma)\Gamma''\bigr)
\]
and
\[
 w=[x_1,x_2^{256}]
   [x_1,x_3^{16384}]
   [x_2,x_3^{32768}].
\]
A computation with the \textsc{GAP} package \textsf{NQ}, reproduced in
\S~\ref{subsec:gap-verification}, gives
that the element $a=[x_4,x_6,x_3,x_3]\in\gamma_4(M)$ has order $2^{14}$ and
$w=a^{2^{13}}$.  Write
$\gamma_4(M)=\langle b_1\rangle\times\cdots\times\langle b_s\rangle$
in Smith normal form and $a=\prod_i b_i^{\alpha_i}$.  The computation
finds an index $j$ such that $b_j$ has order $2^{14}$ and $\alpha_j$ is odd.
For $C_{2^{14}}=\langle z\rangle$, define the epimorphism
\[
 \varphi\colon\gamma_4(M)\longrightarrow C_{2^{14}},\qquad
 \varphi(b_i)=
 \begin{cases}
 z^{\alpha_j^{-1}},&i=j,\\
 1,&i\neq j,
 \end{cases}
\]
where $\alpha_j^{-1}$ is taken modulo $2^{14}$.  Then $\varphi(a)=z$.
Since $\gamma_5(M)=1$, the subgroup $K=\ker\varphi\subseteq\gamma_4(M)$
is central in $M$.  Set $G=M/K$.  Then
\begin{equation}\label{eq:group-plateau-gap}
 \gamma_4(G)=\langle a\rangle\cong C_{2^{14}},\qquad
 a=[x_4,x_6,x_3,x_3],\qquad
 w=a^{2^{13}}\neq1.
\end{equation}
The identities $G''=1$ and $\gamma_5(G)=1$ follow from the definition of $M$.

To verify the dimension-subgroup assertion, put
\[
 \zeta_1=d_3^4d_5^{-4}d_2^{-1}[d_2,x_2]^2,
 \qquad
 \zeta_2=d_3[d_5,x_5]^{-4}.
\]
Expanding the relators~\eqref{eq:group-plateau-presentation} in the group
ring, as in \cite[Theorems~2.14--2.15 and Example~2.16]{MikhailovPassi2009},
gives
\begin{equation}\label{eq:group-augmentation-calculation}
 1-w\equiv
 (1-\zeta_1)(1-x_1^{16})+(1-\zeta_2)(1-x_2^{128})
 \pmod{\Delta(\Gamma)^6}.
\end{equation}
The elements $\zeta_1,\zeta_2$ lie in $\gamma_3(\Gamma)$ because each
$d_i$ does.  The first two defining relators also give
$x_1^8,x_2^{64}\in\gamma_3(\Gamma)$, and hence
$x_1^{16},x_2^{128}\in\gamma_3(\Gamma)$.  Thus both
products in~\eqref{eq:group-augmentation-calculation} lie in
$\Delta(\Gamma)^6$.  Hence $w\in D_6(\Gamma)$, and functoriality gives
$w\in D_6(G)$.

To bound $D_5(G)$ from above, first put $Q=G/\gamma_4(G)$.
The first four relators in
\eqref{eq:group-plateau-presentation} express powers of $x_1,x_2,x_3,x_5$
as elements of $\gamma_3(Q)\subseteq Z(Q)$.  Order the generators as
$(y_1,\ldots,y_6)=(x_3,x_5,x_2,x_1,x_4,x_6)$ and set
$(e_1,e_2,e_3,e_4)=(16384,1024,64,8)$.  The standard description of
$D_4(Q)$ in \cite[Theorem~2.22]{MikhailovPassi2009} expresses its elements
as products of commutators $[y_i^{e_i},y_j]$, with $1\leq i<j\leq4$,
subject to conditions on their exponents.  Each of these commutators is
trivial, since $y_i^{e_i}\in Z(Q)$.  Therefore $D_4(Q)=1$.
Functoriality gives $D_4(G)\subseteq\gamma_4(G)$, and the reverse
inclusion holds for every group.  Hence $D_4(G)=\gamma_4(G)$.
Moreover, the fifth-dimension-subgroup theorem
\cite[Theorem~2.28]{MikhailovPassi2009} gives
\[
                         D_5(G)^2\subseteq G''\gamma_5(G)=1.
\]
Thus $D_5(G)$ is contained in the unique subgroup of order two in the cyclic
group~\eqref{eq:group-plateau-gap}.  Since
$1\neq w\in D_6(G)\subseteq D_5(G)$, we obtain
\begin{equation}\label{eq:group-D5-D6}
                 D_5(G)=D_6(G)=\langle w\rangle\cong C_2.
\end{equation}

For finitely generated metabelian groups the dimension and lower central
series eventually coincide~\cite{GuptaHalesPassi1984}.
Since $G$ is finitely generated and $\gamma_5(G)=1$, there is
therefore a least integer $r$ with $D_r(G)=1$.
By~\eqref{eq:group-D5-D6}, $r\geq7$.
For $5\leq j<r$, the group $D_j(G)$ is a nontrivial subgroup of
$\langle w\rangle\cong C_2$, so it equals $\langle w\rangle$.
Taking $n=r-2$ gives
\[
                  D_n(G)=D_{n+1}(G)\cong C_2,
                  \qquad D_{n+2}(G)=1.
\]
\section{Centrality of dimension subrings}\label{sec:centrality}

\begin{Lemma}\label{lem:adjoint-enveloping-action}
Let $L$ be a Lie ring and $A$ be an ideal of $L$.  The adjoint action extends uniquely to a right
$\U(L)$-module structure on $A$, and
\[
                         A\cdot\varpi(L)^n=[A,{}_nL],\qquad n\geq0.
\]
\end{Lemma}

\begin{proof}
On the tensor algebra of $L$, define
\[
 a\cdot1=a,
 \qquad
 a\cdot(x_1\cdots x_r)=[a,x_1,\ldots,x_r].
\]
For $a\in A$ and $x,y\in L$, the Jacobi identity gives
\[
 [[a,x],y]-[[a,y],x]-[a,[x,y]]=0.
\]
Hence the defining relations of $\U(L)$ annihilate $A$, and the action
factors through $\U(L)$.  Uniqueness follows because $\U(L)$ is generated
as an algebra by the image of $L$.  The ideal $\varpi(L)^n$ is generated
additively by products of images of at least $n$ elements of $L$.  A product of length $r\geq n$
acts into $[A,{}_rL]\subseteq[A,{}_nL]$.  Conversely, every generator
$[a,x_1,\ldots,x_n]$ of $[A,{}_nL]$ is the action of
$x_1 \ldots x_n\in\varpi(L)^n$ on $a$.
\end{proof}

\begin{proof}[Proof of Theorem~\ref{thm:B}]
Let $A$ be an ideal of $L$, let $a\in A$, and let $x\in D_r(L)$.
Lemma~\ref{lem:adjoint-enveloping-action} gives
\[
 [a,x]=a\cdot x
       \in A\cdot\varpi(L)^r
       =[A,{}_rL].
\]
Thus $[A,D_r(L)]\subseteq[A,{}_rL]$.  Taking $A=L$ and using
skew-symmetry gives
$[D_r(L),L]\subseteq\gamma_{r+1}(L)$.  The inclusion
$\gamma_r(L)\subseteq D_r(L)$ gives the reverse containment:
\[
 \gamma_{r+1}(L)=[\gamma_r(L),L]
                  \subseteq[D_r(L),L].
\]
Hence $[D_r(L),L]=\gamma_{r+1}(L)$.  Since $D_n(L)$ is the
inverse image of the two-sided ideal $\varpi(L)^n$, it is an ideal of
$L$.  Applying $[A,D_r(L)]\subseteq[A,{}_rL]$ with $A=D_n(L)$ and $r=m$
gives
\[
 \begin{aligned}
 [D_n(L),D_m(L)]
 &\subseteq[D_n(L),{}_mL]\\
 &=[\gamma_{n+1}(L),{}_{m-1}L]
  =\gamma_{n+m}(L).
 \end{aligned}\qedhere
\]
\end{proof}

\section{A vanishing criterion for \texorpdfstring{$D_n(L)$}{D\_n(L)}}
\label{sec:relative-vanishing}

In this section, we describe a method for proving that $D_n(L)=0$ for a
Lie ring $L$ whose additive group is finitely generated. We will apply
this method to prove Theorem~\ref{thm:C} in~\S~\ref{sec:metabelian_lie_rings}.

Our method is to find an auxiliary object, which we call a
\textit{certificate}, whose properties imply that $D_n(L)=0$. Such a
certificate is a pair $(T,p)$, where $T$ is a central extension of $L$
and
\[
p\colon \U(T)\twoheadrightarrow P_{\leq 2}\U(T)
\]
is an additive retraction satisfying an additional condition. Before
stating this condition, observe that a homomorphism of Lie rings
$\pi\colon T\to L$ induces a homomorphism of filtered algebras
\[
\pi_*\colon \U(T)\longrightarrow \U(L).
\]
For additive subgroups $A,B\subseteq\U(T)$, we write
\(
A\cdot B=\langle ab\mid a\in A,\ b\in B\rangle
\)
for the additive subgroup generated by their products.
We can now state the main criterion.

\begin{Theorem}
\label{thm:vanishing-criterion}
Let $n\geq 1$ be an integer, and let
\[
0\longrightarrow K\longrightarrow T\xrightarrow{\pi}L
\longrightarrow 0
\]
be a central extension of Lie rings such that the additive group of $T$
is finitely generated. Let
\[
p\colon \U(T)\twoheadrightarrow P_{\leq 2}\U(T)
\]
be an additive retraction of the inclusion
\(P_{\leq 2}\U(T)\hookrightarrow\U(T)\) such that
\begin{equation}
\label{eq:vanishing-criterion-condition}
p\left(\pi_*^{-1}(\varpi(L)^n)\right)
\subseteq K+Z(T)\cdot T.
\end{equation}
Then \(D_n(L)=0\).
\end{Theorem}

We next state a technical lemma, derive
Theorem~\ref{thm:vanishing-criterion} from it, and then prove the lemma.

\begin{Lemma}
\label{lem:quadratic-retraction}
Let $T$ be a Lie ring whose additive group is finitely generated. Then the
inclusion $T\hookrightarrow P_{\leq 2}\U(T)$ admits an additive
retraction
\(
r\colon P_{\leq 2}\U(T)\twoheadrightarrow T
\)
such that $r|_{Z(T)\cdot T}=0$.
\end{Lemma}

\begin{proof}[Proof of Theorem~\ref{thm:vanishing-criterion}]
Let $x\in D_n(L)$, and choose a lift $\widetilde{x}\in T$ along $\pi$.
Then
\(
\widetilde{x}\in\pi_*^{-1}\left(\varpi(L)^n\right)\).
Let $r$ be the retraction from Lemma~\ref{lem:quadratic-retraction}.
Since $p$ and $r$ restrict to the identity on $T$,
condition~\eqref{eq:vanishing-criterion-condition} gives
\[
\widetilde{x}=r\bigl(p(\widetilde{x})\bigr)
\in r\bigl(K+Z(T)\cdot T\bigr)=K.
\]
Thus $x=\pi(\widetilde{x})=0$.
\end{proof}

\begin{proof}[Proof of Lemma~\ref{lem:quadratic-retraction}]
We write bars for classes modulo $Z(T)$. We will
construct a biadditive map
\[
H\colon T/Z(T)\times T/Z(T)\longrightarrow T
\]
such that
\begin{equation}
\label{eq:condition_on_H}
H(\bar{x},\bar{y})-H(\bar{y},\bar{x})
=[x,y],\qquad x,y\in T.
\end{equation}

We first show that any such map $H$ determines the required retraction
$r$. Since the additive group of $T$ is finitely generated, there is
a decomposition \(
T=\bigoplus_{i=1}^k\langle x_i\rangle
\) into a direct sum of cyclic subgroups.
Order the chosen generators by $x_1<\cdots<x_k$.
By Corollary~\ref{cor:pbw-basis}, there is a direct sum decomposition
\[
P_{\leq 2}\U(T)
=\Z\cdot 1
\oplus\underbrace{\langle x_i\mid 1\leq i\leq k\rangle}_{=T}
\oplus\langle x_i x_j\mid 1\leq i\leq j\leq k\rangle.
\]
Set
\begin{equation}
\label{eq:r_def}
r(1)=0,\qquad
r(x_i)=x_i,\qquad
r(x_i x_j)=H(\bar{x}_i,\bar{x}_j),\qquad i\leq j.
\end{equation}
Suppose that elements $x_i$ and $x_j$ of $T$ have orders $d_i$
and $d_j$, respectively. The element $x_i x_j$ has additive order
$d=\gcd(d_i,d_j)$ in $\U(T)$. Choose integers $a,b$ such that
$d=a d_i+b d_j$. By biadditivity,
\begin{align*}
dH(\bar{x}_i,\bar{x}_j)
&=(a d_i+b d_j)H(\bar{x}_i,\bar{x}_j)\\
&=aH(d_i\bar{x}_i,\bar{x}_j)
  +bH(\bar{x}_i,d_j\bar{x}_j)=0.
\end{align*}
Thus
\eqref{eq:r_def} defines an additive homomorphism
\(
r\colon P_{\leq 2}\U(T)\longrightarrow T
\)
whose restriction to $T$ is the identity.

For $i\leq j$, relation \eqref{eq:condition_on_H} gives
\begin{align*}
r(x_jx_i)
&=r(x_ix_j+[x_j,x_i])\\
&=H(\bar{x}_i,\bar{x}_j)+[x_j,x_i]\\
&=H(\bar{x}_i,\bar{x}_j)
  +\left(H(\bar{x}_j,\bar{x}_i)-H(\bar{x}_i,\bar{x}_j)\right)\\
&=H(\bar{x}_j,\bar{x}_i).
\end{align*}
It follows by biadditivity that
\[
r(xy)=H(\bar{x},\bar{y}),\qquad x,y\in T.
\]
If $z\in Z(T)$, then $\bar z=0$, and hence
$r(zt)=H(\bar z,\bar t)=0$ for every $t\in T$. Therefore
$r$ vanishes on $Z(T)\cdot T$ as required.

It remains to construct $H$. Choose a decomposition \(
T/Z(T)=\bigoplus_{i=1}^s\langle\bar v_i\rangle
\) of the additive group
of $T/Z(T)$
into cyclic groups, and choose a lift $v_i\in T$ of each $\bar v_i$.
Set
\begin{equation}\label{eq:H_definition}
H(\bar v_i,\bar v_j)=
\begin{cases}
[v_i,v_j],& i>j,\\
0,& i\leq j.
\end{cases}
\end{equation}
Since the Lie bracket vanishes whenever one of its arguments is central,
this formula is independent of the chosen lifts. It also respects the
cyclic relations: if $e_i\bar v_i=e_j\bar v_j=0$ in $T/Z(T)$, then
$e_i v_i,e_j v_j\in Z(T)$, and consequently
\begin{align*}
e_i[v_i,v_j]&=[e_i v_i,v_j]=0,\\
e_j[v_i,v_j]&=[v_i,e_j v_j]=0.
\end{align*}
Thus the formula~\eqref{eq:H_definition} defines a biadditive
map $H$. Condition \eqref{eq:condition_on_H} holds by construction.
\end{proof}

\subsection{A convenient condition on certificates for nilpotent Lie rings}

We record a convenient condition on a pair $(T,p)$ that guarantees that
Theorem~\ref{thm:vanishing-criterion} applies. For $r\geq 1$, set
\[
W_r(T)=\gamma_r(T)
+\sum_{i=1}^{r-1}\gamma_i(T)\cdot\gamma_{r-i}(T)
\subseteq P_{\leq 2}\U(T).
\]
Note that
\begin{equation}
\label{eq:W_q_inclusion}
W_q(T)\subseteq W_r(T), \quad q \geq r\geq2.
\end{equation}
Indeed,
$\gamma_q(T)\subseteq\gamma_r(T)$, and for $i+(q-i)=q$ one can choose
$a,b\geq 1$ such that
\[
a+b=r,
\qquad a\leq i,
\qquad b\leq q-i,
\]
so that
\[
\gamma_i(T)\cdot\gamma_{q-i}(T)
\subseteq\gamma_a(T)\cdot\gamma_b(T)
\subseteq W_r(T).
\]

\begin{Corollary}
\label{cor:nilpotent-certificate}
Let $c\geq 2$ be an integer, and let
\[
0\longrightarrow K\longrightarrow T\xrightarrow{\pi}L
\longrightarrow 0
\]
be a central extension of Lie rings such that the additive group of $T$
is finitely generated and $\gamma_{c+1}(T)=0$. Let
\[
p\colon \U(T)\twoheadrightarrow P_{\leq 2}\U(T)
\]
be an additive retraction such that
\begin{align}
p\left(\varpi(T)^{2c-1}\right)
&\subseteq W_{2c-1}(T),
\label{eq:p-augmentation-condition}\\
p\left(K\cdot \U(T)\right)
&\subseteq K+Z(T)\cdot T.
\label{eq:p-kernel-condition}
\end{align}
Then
\(D_{2c-1}(L)=0\).
\end{Corollary}

\begin{proof}
The kernel of the induced epimorphism
\(
\pi_*\colon \U(T)\twoheadrightarrow \U(L)
\)
is $K\cdot \U(T)$, since $K\subseteq Z(T)$. Consequently,
\[
\pi_*^{-1}\left(\varpi(L)^n\right)
=\varpi(T)^n+K\cdot \U(T),\qquad n\geq 1.
\]

Condition \eqref{eq:p-augmentation-condition} gives
\begin{equation}
\label{eq:inclusion_in_wr_subgroup}
p\left(\varpi(T)^{2c-1}\right)
\subseteq\gamma_{2c-1}(T)
+\sum_{\substack{i,j\geq 1\\i+j=2c-1}}
\gamma_i(T)\cdot\gamma_j(T).
\end{equation}
Since $2c-1\geq c+1$, we have $\gamma_{2c-1}(T)=0$. For every
quadratic term $\gamma_i(T)\cdot\gamma_j(T)$ in the sum in
\eqref{eq:inclusion_in_wr_subgroup}, either $i\geq c$ or $j\geq c$.
The corresponding factor is therefore contained in
$\gamma_c(T)\subseteq Z(T)$. Hence
\[
W_{2c-1}(T)\subseteq Z(T)\cdot T.
\]
Conditions \eqref{eq:p-augmentation-condition} and
\eqref{eq:p-kernel-condition} now imply
\eqref{eq:vanishing-criterion-condition} for $n=2c-1$.
Theorem~\ref{thm:vanishing-criterion} completes the proof.
\end{proof}

\section{Metabelian nilpotent Lie rings}
\label{sec:metabelian_lie_rings}
The purpose of this section is to prove the following theorem.

\begin{Theorem}
\label{thm:main_thm_for_finitely_generated}
Let $c\geq 2$ be an integer, and let $L$ be a metabelian Lie ring such that the
additive group of $L$ is finitely generated and
$\gamma_{c+1}(L)=0$. Then
\(D_{2c-1}(L)=0\).
\end{Theorem}

Theorem~\ref{thm:main_thm_for_finitely_generated} differs from
Theorem~\ref{thm:C} only by the finite generation assumption. The
general case of Theorem~\ref{thm:C} follows from
Theorem~\ref{thm:main_thm_for_finitely_generated} by the standard
argument expressing an arbitrary Lie ring as a filtered colimit of
finitely generated Lie subrings. We carry out this passage at the end of
the section.

Our strategy for proving
Theorem~\ref{thm:main_thm_for_finitely_generated} consists of the
following two steps.

\textit{Step 1.}
We construct a specific central extension
\[
0\longrightarrow K\longrightarrow T\xrightarrow{\pi}L
\longrightarrow 0
\]
with the following properties:
\begin{enumerate}
\item for every $i\geq 2$, the projection $\pi$ induces an isomorphism
\(\gamma_i(T) \cong \gamma_i(L)\);
\item the derived subring $T'$, canonically isomorphic to $L'$, is an
additive direct summand of $T$.
\end{enumerate}

\textit{Step 2.}
We choose a suitable decomposition of the additive group of $T$ into a
direct sum of cyclic groups. Using this decomposition, we construct an
additive retraction
\(
p\colon \U(T)\twoheadrightarrow P_{\leq 2}\U(T)
\)
satisfying the hypotheses of
Corollary~\ref{cor:nilpotent-certificate}. It then follows immediately
that $D_{2c-1}(L)=0$.

The role of the index $2c-1$ is already visible in the proof of
Corollary~\ref{cor:nilpotent-certificate}. It is the least integer $n\geq 2$
such that, for every decomposition $n=i+j$ with $i,j\geq 1$, either
$i\geq c$ or $j\geq c$. We apply this observation to the quadratic part
of the decomposition in $\U(T)$ of an element of $\varpi(T)^n$. This part is
expressed as a sum of products belonging to
$\gamma_i(T)\cdot\gamma_j(T)$ with $i+j\geq n$. Since
$\max\{i,j\}\geq c$ and $\gamma_c(T)\subseteq Z(T)$, every quadratic
term belongs to $Z(T)\cdot T$ and can therefore be controlled by the
vanishing criterion.

\subsection{Step 1: construction of \texorpdfstring{$T$}{T}}
\label{subsec:construction-of-T}

Let $L$ be a metabelian Lie ring whose additive group is finitely
generated. Its abelianization $L_{\ab}=L/L'$ is a finitely
generated abelian group, which we regard as a Lie ring with trivial Lie
bracket. Choose an epimorphism
\(
\rho\colon P\twoheadrightarrow L_{\ab}
\)
from a finitely generated free abelian group $P$, also regarded as a Lie
ring with trivial Lie bracket. Set
\begin{equation}
\label{eq:T_definition}
T=L\times_{L_{\ab}}P
=\left\{(x,v)\in L\times P\mid x+L'=\rho(v)\right\}.
\end{equation}
The Lie bracket on $T$ is given by
\begin{equation}
\label{eq:T_lie_bracket}
[(x,v),(y,w)]=([x,y],0).
\end{equation}
Let $\pi\colon T\twoheadrightarrow L$ be the canonical projection, and
set $K=\ker(\pi)$.

\begin{Lemma}
\label{lem:properties-of-T}
The extension
\[
0\longrightarrow K\longrightarrow T\xrightarrow{\pi}L
\longrightarrow 0
\]
is central and has the following properties:
\begin{enumerate}
\item for every $i\geq 2$, the projection $\pi$ induces an isomorphism
\(
\gamma_i(T)\cong\gamma_i(L)\);
\item the canonical projection $T\twoheadrightarrow P$ fits into a
split short exact sequence of abelian groups
\[
0\longrightarrow T'\longrightarrow T\longrightarrow P
\longrightarrow 0.
\]
Consequently,
\[T\cong T'\oplus P \; (\cong L'\oplus P)\]
noncanonically as abelian groups.
\end{enumerate}
\end{Lemma}

\begin{proof}
The extension $T\twoheadrightarrow L$ is obtained by pulling back the
central extension
\[
0\longrightarrow\ker(\rho)\longrightarrow P
\xrightarrow{\rho}L_{\ab}\longrightarrow 0
\]
along the canonical projection $L\twoheadrightarrow L_{\ab}$.
Hence it is central.

Formula \eqref{eq:T_lie_bracket} gives
\(
T'=L'\times\{0\}
\)
and, inductively,
\(
\gamma_i(T)=\gamma_i(L)\times\{0\}\) for each \(i\geq 2\).
This proves the first assertion.

For the second assertion, consider the following diagram, whose rows
and columns are exact:
\[
\begin{tikzcd}[cramped]
	&& {T'} & {L'} \\
	0 & K & T & L & 0 \\
	0 & {\ker(\rho)} & P & {L_{\ab}} & 0
	\arrow["\cong"{description}, draw=none, from=1-3, to=1-4]
	\arrow[hook, from=1-3, to=2-3]
	\arrow[hook, from=1-4, to=2-4]
	\arrow[from=2-1, to=2-2]
	\arrow[from=2-2, to=2-3]
	\arrow["\cong"{marking, allow upside down}, draw=none, from=2-2, to=3-2]
	\arrow[from=2-3, to=2-4]
	\arrow[two heads, from=2-3, to=3-3]
	\arrow["\lrcorner"{anchor=center, pos=0.125}, draw=none, from=2-3, to=3-4]
	\arrow[from=2-4, to=2-5]
	\arrow[two heads, from=2-4, to=3-4]
	\arrow[from=3-1, to=3-2]
	\arrow[from=3-2, to=3-3]
	\arrow[from=3-3, to=3-4]
	\arrow[from=3-4, to=3-5]
\end{tikzcd}
\]
Both isomorphisms `$\cong$' marked in the diagram follow from the fact that
kernels are preserved under pullback. In particular, the canonical
projection $T\twoheadrightarrow P$ has kernel $T'$. The resulting short
exact sequence of abelian groups
\[
0\longrightarrow T'\longrightarrow T\longrightarrow P
\longrightarrow 0
\]
splits because $P$ is free abelian.
\end{proof}

Since we will work explicitly with a PBW basis, it is convenient to fix
specific decompositions into direct sums of cyclic groups for the
abelian groups involved. Choose an invariant factor decomposition
\begin{equation}
\label{eq:L_ab_decomposition}
L_{\ab}\cong
\bigoplus_{i=1}^m\Z/d_i\Z,
\end{equation}
where
\[
d_i\geq 0,\qquad 1\leq i\leq m,
\qquad
d_1\mid\cdots\mid d_m.
\]
As usual, $\Z/0\Z=\Z$.

We now specify the choice of the free abelian group $P$. Set
\begin{equation}
\label{eq:P_decomposition}
P=\bigoplus_{i=1}^m\Z e_i,
\end{equation}
and define
\(
\rho\colon P\twoheadrightarrow L_{\ab}
\)
by sending $e_i$ to the standard generator of the $i$th cyclic summand
in \eqref{eq:L_ab_decomposition}. Then
\[
\ker(\rho)
=\langle d_i e_i\mid 1\leq i\leq m\rangle
\subseteq P.
\]

\begin{Corollary}
\label{cor:adapted-decomposition-of-T}
For the central extension
\[
0\longrightarrow K\longrightarrow T\xrightarrow{\pi}L
\longrightarrow 0
\]
constructed above, the Lie ring $T$ is metabelian, and there exist an
integer $k\geq 0$ and elements
\[
t_1,\ldots,t_k\in T',
\qquad
p_1,\ldots,p_m\in T,
\]
such that, additively,
\[
T'=\bigoplus_{i=1}^k\langle t_i\rangle,
\qquad
T=T'\oplus\bigoplus_{j=1}^m\Z p_j.
\]
Moreover, there exist elements $b_1,\ldots,b_m\in T'$ such that the
subgroup $K\subseteq T$ is generated by
\[
d_i p_i+b_i,\qquad 1\leq i\leq m.
\]
\end{Corollary}

\begin{proof}
This follows directly from Lemma~\ref{lem:properties-of-T}. Since
$T'\cong L'$ and $L$ is metabelian, the Lie ring $T$ is metabelian.
Moreover, $T$ is an additive subgroup of $L\times P$. Since the additive
groups of $L$ and $P$ are finitely generated, the additive group of $T$
is finitely generated. Hence the additive group of $T'$ is finitely
generated, and we may choose a decomposition
\(
T'=\bigoplus_{i=1}^k\langle t_i\rangle
\)
into cyclic groups.

We now describe explicit choices of the elements $p_i\in T$ and
$b_i\in T'$. We use the decomposition \eqref{eq:P_decomposition} and
the realization $T\subseteq L\times P$ from \eqref{eq:T_definition}.
For each $i$, choose a lift $a_i\in L$ of
$\rho(e_i)\in L_{\ab}$, and set
\(p_i=(a_i,e_i)\in T\).
Since $P$ is free on $e_1,\ldots,e_m$, the assignment $e_i\mapsto p_i$
extends to an additive section of the projection $T\twoheadrightarrow P$. It follows
that
\[
T=T'\oplus\bigoplus_{i=1}^m\Z p_i.
\]
Under the identification $K\cong\ker(\rho)$, the subgroup $K$ is
generated by the elements $(0,d_i e_i)$. Since
$d_i\rho(e_i)=0$ in $L_{\ab}$, we have $d_i a_i\in L'$, and
therefore
\[
(0,d_i e_i)
=d_i p_i+\underbrace{(-d_i a_i,0)}_{=:b_i\in T'},
\qquad 1\leq i\leq m.
\]
This proves the assertion.
\end{proof}

\subsection{Step 2: the additive retraction
\texorpdfstring{$p\colon \U(T)\twoheadrightarrow P_{\leq 2}\U(T)$}
{the additive retraction onto filtration degree at most two}}
\label{subsec:additive-retraction}

By Corollary~\ref{cor:adapted-decomposition-of-T}, there is a subset
\[
\mathcal X=
\left\{
\underbrace{t_1,\ldots,t_k}_{\text{generators of}\;T'},
\underbrace{p_1,\ldots,p_m}_{\text{generators of}\;P}
\right\} \subseteq T,
\]
and there are elements $b_1,\ldots,b_m\in T'$ such that
\[
T=\bigoplus_{x\in\mathcal X}\langle x\rangle,
\qquad
K=\langle d_i p_i+b_i\mid 1\leq i\leq m\rangle.
\]
Equip $\mathcal X$ with the total order
\[
t_1<\cdots<t_k<p_1<\cdots<p_m.
\]
As before, set $M_0=\Z\cdot 1$ and, for $q\geq 1$, let
\[
M_q=
\langle x_1\cdots x_q
\mid x_i\in\mathcal X,\ x_1\leq\cdots\leq x_q\rangle
\subseteq \U(T).
\]
There is an additive projection
\[
p\colon
\U(T)=\bigoplus_{r\geq 0}M_r
\twoheadrightarrow
P_{\leq 2}\U(T)=M_0\oplus M_1\oplus M_2
\]
defined by
\[
\left.p\right|_{M_r}=
\begin{cases}
\operatorname{id}_{M_r},& r\leq 2,\\
0,& r>2.
\end{cases}
\]
We will prove that this projection satisfies the hypotheses of
Corollary~\ref{cor:nilpotent-certificate}. We first establish several
preliminary results.

\begin{Lemma}
\label{lem:central-multiple-commutator}
Let $1\leq i\leq m$, and let $x\in\mathcal X$ satisfy $x<p_i$. Then
\(d_i[p_i,x]\in Z(T)\).
\end{Lemma}

\begin{proof}
Since $d_i p_i+b_i\in K$ and $b_i\in T'$, we have
\[
d_i p_i\in K+T'\subseteq Z(T)+T'.
\]
Moreover,
\(
[Z(T)+T',T']=0
\)
because $T$ is metabelian.

\emph{Case 1.} Suppose that $x=t_j$ for some $1\leq j\leq k$. Then
\[
d_i[p_i,t_j]=[d_i p_i,t_j]\in[Z(T)+T',T']=0.
\]

\emph{Case 2.} Suppose that $x=p_j$ for some $1\leq j<i$. For every
$y\in T$, the Jacobi identity gives
\begin{align*}
d_i[[p_i,p_j],y] =
- \Big[
\underbrace{[p_j,y]}_{\in T'},
\underbrace{d_i p_i}_{\in Z(T)+T'}
\Big]
- \Big[
\underbrace{[y, p_i]}_{\in T'},
\underbrace{d_i p_j}_{\substack{\in Z(T)+T'\\ \text{since }d_j\mid d_i}}
\Big] = 0.
\end{align*}
Thus $d_i[p_i,p_j]\in Z(T)$.
\end{proof}

The following lemmas~\ref{lem:multiplication-by-generators}
and~\ref{lem:ordered-products-with-commutators} establish properties of a
rewriting system that expresses every element of $\U(T)$ in terms
of its components in the decomposition
\(
\U(T)=\bigoplus_{r=0}^{\infty}M_r
\).

\begin{Lemma}
\label{lem:multiplication-by-generators}
\leavevmode
\begin{enumerate}
\item Let $q\geq0$ be an integer. Then
\[
T'\cdot M_q\subseteq M_{q+1}.
\]
\item Let $i,q$ be integers with $1\leq i\leq m$ and $q\geq1$. Then
\[
d_i p_i\cdot M_q
\subseteq M_{q+1}+(Z(T)\cap T')\cdot M_{q-1}
\subseteq M_{q+1}+M_q.
\]
\end{enumerate}
\end{Lemma}

\begin{proof}
Let
\[
u=x_1\cdots x_q\in M_q,
\qquad
x_\nu\in\mathcal X,
\qquad
x_1\leq\cdots\leq x_q,
\]
be a generator of $M_q$. There exists $s$ with $0\leq s\leq q$ such
that
\[
\{x_1,\ldots,x_s\}\subseteq\{t_1,\ldots,t_k\},
\qquad
\{x_{s+1},\ldots,x_q\}\subseteq\{p_1,\ldots,p_m\}.
\]
Thus
\[
u=t_{j_1}\cdots t_{j_s}p_{i_1}\cdots p_{i_{q-s}}
\]
for some
\[
1\leq j_1\leq\cdots\leq j_s\leq k,
\qquad
1\leq i_1\leq\cdots\leq i_{q-s}\leq m.
\]

To prove (1), it suffices by additivity to consider $t_i u$ with
$1\leq i\leq k$. Since $T$ is metabelian, the elements $t_i$ and $t_j$
commute in $\U(T)$. Hence $t_i u$ can
be placed in nondecreasing order without producing correction terms.
Since every $t_j$ precedes every $p_i$ in $\mathcal X$, the resulting
ordered monomial belongs to $M_{q+1}$.

To prove (2), assume $q\geq 1$ and fix $1\leq i\leq m$. Starting with
\[
d_i p_i u
=d_i p_i t_{j_1}\cdots t_{j_s}p_{i_1}\cdots p_{i_{q-s}},
\]
we move $d_i p_i$ to its position in the nondecreasing order. At each step
we use the rewriting rule
\[
(\dots d_ip_i \cdot x \dots)
\rightsquigarrow (\dots x \cdot d_i p_i \dots) + \underbrace{(\dots d_i [p_i,x]\dots)}_{\text{correction term}},
\qquad x \in \mathcal{X}, x<p_i.
\]
The ordered term obtained after all rewrites belongs to $M_{q+1}$.

It remains to consider correction terms. By
Lemma~\ref{lem:central-multiple-commutator},
$d_i[p_i,x]\in Z(T)$. Since $d_i[p_i,x]\in T'$, we have
\[
\begin{aligned}
\text{(correction term)}
&=u_1\cdot d_i[p_i,x]\cdot u_2
=d_i[p_i,x]\cdot u_1u_2\\
&\in (Z(T)\cap T')\cdot M_{q-1}\subseteq M_q,
\end{aligned}
\]
where the equality uses the centrality of $d_i[p_i,x]$ and the
last inclusion follows from the first assertion.
Here $u_1u_2$ is obtained from the ordered monomial $u$ by deleting the
occurrence of $x$ involved in the rewriting step; hence $u_1u_2\in M_{q-1}$.
Therefore
\[
d_i p_i u\in M_{q+1}+(Z(T)\cap T')\cdot M_{q-1}
\subseteq M_{q+1}+M_q,
\]
which proves (2).
\end{proof}

\begin{Lemma}
\label{lem:ordered-products-with-commutators}
Let $q\geq1$ be an integer and let $x_1,\ldots,x_q\in\mathcal X$. The product
$x_1\cdots x_q\in \U(T)$ is a $\Z$-linear combination of products
of the form
\begin{equation}
\label{eq:ordered-product-with-commutators}
c_1\cdots c_h y_1\cdots y_v\in M_{h+v},
\qquad
y_j\in\mathcal X,
\qquad
y_1\leq\cdots\leq y_v,
\end{equation}
where $h,v\geq 0$, each $c_\nu$ is a left-normed commutator of length
$s_\nu\geq 2$ in elements of $\mathcal X$, and
\[
s_1+\cdots+s_h+v=q.
\]
\end{Lemma}
\begin{proof}
We first show that the product $x_1\cdots x_q$ can be rewritten as a
$\Z$-linear combination of products of the form
\[
f_1\cdots f_r p_{i_1}\cdots p_{i_\ell},
\qquad
i_1\leq\cdots\leq i_\ell,
\]
where each $f_j$ is either a left-normed commutator of length at least
$2$ in elements of $\mathcal X$ or one of the elements
$t_1,\ldots,t_k$. To this end, we repeatedly apply the rewriting rule
\[
(\dots p_i\cdot a\dots)
\rightsquigarrow
(\dots a\cdot p_i\dots) - (\dots[a,p_i]\dots),
\]
where $a$ is either a left-normed commutator of weight at least $2$, an
element $t_\alpha$, or an element $p_j$ with $j<i$. The new commutator
$[a,p_i]$ is again left-normed and has weight one greater than that of
$a$, where a letter is regarded as a commutator of weight $1$.

An inversion is a pair of positions, not necessarily adjacent, whose
left factor is $p_i$ and whose right factor is a commutator of weight at
least $2$, a letter $t_\alpha$, or a letter $p_j$ with $j<i$.
In the first term on the right-hand side, swapping the adjacent inverted
pair decreases the number of inversions by one without changing the
number of factors; the second term has fewer factors.
Thus the lexicographically ordered pair consisting of the number of
factors and the number of inversions decreases at every branch, so the
rewriting process terminates.

Every factor $f_j$ belongs to $T'$. Since $T$ is metabelian, the
elements $f_1,\ldots,f_r$ commute with one another in
$\U(T)$. We may therefore place the commutators of length at
least $2$ first, in an arbitrary order, and then arrange the remaining
letters $t_\alpha$ in nondecreasing order. Since every $t_\alpha$
precedes every $p_i$ in $\mathcal X$, each resulting term has the form
\[
c_1\cdots c_h y_1\cdots y_v,
\qquad
y_1\leq\cdots\leq y_v,
\]
appearing in \eqref{eq:ordered-product-with-commutators}.

The sum of the weights
is preserved by each rewriting rule. Since the original product has total weight $q$, every
resulting term satisfies
\[
s_1+\cdots+s_h+v=q.
\]
Finally, $y_1\cdots y_v\in M_v$ and each $c_\nu\in T'$. Repeated
application of Lemma~\ref{lem:multiplication-by-generators}~(1) gives
\[
c_1\cdots c_h y_1\cdots y_v\in M_{h+v}.
\]
This proves the assertion.
\end{proof}

\begin{Proposition}
\label{prop:projection-properties}
The projection $p$ has the following properties:
\begin{enumerate}
\item for every $r\geq 2$,
\[
p\left(\varpi(T)^r\right)\subseteq W_r(T);
\]
\item one has
\[
p\left(K\cdot \U(T)\right)\subseteq K+Z(T)\cdot T.
\]
\end{enumerate}
\end{Proposition}

\begin{proof}
(1) The products
\[
x_1\cdots x_q,
\qquad q\geq r,
\qquad x_i\in\mathcal X,
\]
generate $\varpi(T)^r$ additively. It is therefore enough to prove that
\(p(x_1\cdots x_q)\in W_r(T)\).
By Lemma~\ref{lem:ordered-products-with-commutators}, we may express
$x_1\cdots x_q$ as a $\Z$-linear combination of terms of the
form \eqref{eq:ordered-product-with-commutators}. The projection $p$
annihilates every such term lying in $M_{h+v}$ with $h+v>2$.
Consequently, it remains to consider the following six possibilities:
\[
\begin{array}{c|l@{\qquad}r}
\eqref{eq:ordered-product-with-commutators}\in M_{h+v}
  & \multicolumn{2}{c}{\text{summand}} \\
\hline
M_{0+0} & 1
  & \text{(impossible since }q\geq 2\text{)} \\
M_{0+1} & y_1
  & \text{(impossible since }q\geq 2\text{)} \\
M_{0+2} & y_1y_2
  & \text{(possible only when }q=2\text{)} \\
M_{1+0} & c_1\in\gamma_{s_1}(T),
  & s_1=q \\
M_{1+1} & c_1y_1\in\gamma_{s_1}(T)\cdot T,
  & s_1+1=q \\
M_{2+0} & c_1c_2\in
  \gamma_{s_1}(T)\cdot\gamma_{s_2}(T),
  & s_1+s_2=q
\end{array}
\]
The table shows that
\[
p(x_1\cdots x_q)
\in\gamma_q(T)
+\sum_{i=1}^{q-1}\gamma_i(T)\cdot\gamma_{q-i}(T)
=W_q(T) \overset{\eqref{eq:W_q_inclusion}}{\subseteq} W_r(T).
\]
This proves (1).

(2) Fix an additive generator
\(
d_i p_i+b_i
\)
of $K$, where $1\leq i\leq m$, and an ordered monomial
\[
x_1\cdots x_q\in M_q,
\qquad q\geq 0,
\qquad x_j\in\mathcal X,
\qquad x_1\leq\cdots\leq x_q.
\]
We first prove that
\[
p\left((d_i p_i+b_i)\cdot x_1 \ldots x_q\right)
\in K+Z(T)\cdot T.
\]
We consider the following cases.

\begin{itemize}
\item If $q=0$, then
\(d_i p_i + b_i \in K \subseteq T \subseteq P_{\le 2}\U(T)\), and hence
\[
p(d_i p_i+b_i)=d_i p_i+b_i\in K.
\]

\item If $q=1$, then $(d_i p_i+b_i) x_1 \in P_{\leq 2}\U(T)$, and hence
\[
p\left((d_i p_i+b_i)x_1\right)
=(d_i p_i+b_i)x_1\in K \cdot T \subseteq Z(T)\cdot T.
\]

\item Suppose that $q=2$. By Lemma~\ref{lem:multiplication-by-generators}~(2),
\[
d_i p_i\cdot x_1x_2\in M_3+Z(T)\cdot T.
\]
Moreover, Lemma~\ref{lem:multiplication-by-generators}~(1) gives
\[
b_i \cdot x_1 x_2 \in M_3,
\]
because $b_i\in T'$. Thus
\[
p\left((d_i p_i+b_i) \cdot x_1 x_2\right)\in Z(T)\cdot T.
\]

\item Suppose that $q\geq 3$. Since $b_i\in T'$, the two assertions of
Lemma~\ref{lem:multiplication-by-generators} give
\[
d_i p_i \cdot x_1 \ldots x_q \in M_{q+1}+M_q,
\qquad
b_i \cdot x_1 \ldots x_q \in M_{q+1}.
\]
Both $M_q$ and $M_{q+1}$ are annihilated by $p$, so
\[
p\left((d_i p_i+b_i) \cdot x_1 \ldots x_q\right)=0.
\]
\end{itemize}

The elements $d_i p_i+b_i$, $1\leq i\leq m$, generate $K$
additively, while the ordered monomials in the groups $M_q$, $q\geq 0$,
generate $\U(T)$ additively. The required inclusion
\[
p\left(K\cdot \U(T)\right) = p\left(\sum_{i,q} (d_ip_i + b_i) \cdot M_q\right) \subseteq K+Z(T)\cdot T
\]
now follows by additivity.
\end{proof}

\subsection{Proof of Theorem~\ref{thm:C} and its consequences}

\begin{proof}[Proof of
Theorem~\ref{thm:main_thm_for_finitely_generated}]
Let $c\geq 2$ be an integer, and let $L$ be a metabelian Lie ring such that the
additive group of $L$ is finitely generated and
$\gamma_{c+1}(L)=0$. Consider the central extension
\[
0\longrightarrow K\longrightarrow T\xrightarrow{\pi}L
\longrightarrow 0
\]
constructed in \S~\ref{subsec:construction-of-T}, together with
the additive retraction
\(
p\colon \U(T)\twoheadrightarrow P_{\leq 2}\U(T)
\)
constructed in \S~\ref{subsec:additive-retraction}. By
Lemma~\ref{lem:properties-of-T}~(1), we have
$\gamma_{c+1}(T)=0$. By
Corollary~\ref{cor:adapted-decomposition-of-T}, the additive group of
$T$ is finitely generated. Proposition~\ref{prop:projection-properties}
verifies the two conditions required in
Corollary~\ref{cor:nilpotent-certificate}. Therefore that corollary
gives
\(
D_{2c-1}(L)=0\).
\end{proof}

To treat arbitrary $L$, we reduce to finitely generated Lie subrings.

\begin{Lemma}\label{lem:finitely-generated-reduction}
Let $n\geq1$ be an integer, and let $L$ be a Lie ring. If $D_n(L)\neq0$, then there is a finitely generated Lie
subring $L_0\subseteq L$ such that $D_n(L_0)\neq0$.
\end{Lemma}

\begin{proof}
Choose $0\neq z\in D_n(L)$ and express it in $\U(L)$ as an integral
linear combination of products of at least $n$ elements of $L$.
Since $\U$ preserves filtered colimits, this finite identity holds in
$\U(L_0)$ for some finitely generated Lie subring $L_0\subseteq L$
containing $z$ and all the factors.  Hence $0\neq z\in D_n(L_0)$.
\end{proof}

\begin{proof}[Proof of Theorem~\ref{thm:C}]
By Lemma~\ref{lem:finitely-generated-reduction}, it suffices to
consider finitely generated $L$, since Lie subrings inherit metabelianity
and the bound on the nilpotency class.  The additive group of such an $L$
is generated by Lie monomials of length at most $c$ in a finite generating
set, and is therefore finitely generated.
The result follows from Theorem~\ref{thm:main_thm_for_finitely_generated}.
\end{proof}

\begin{Corollary}\label{cor:general-odd}
\leanproof{\leanrepository/LieRings/Publication/Results.lean\#L207-L215}
Let $c \ge 2$ be an integer, and let $L$ be a Lie ring.
Then
\[
              D_{2c-1}(L)\subseteq\gamma_{c+1}(L)+L''.
\]
In particular, $D_5(L)\subseteq\gamma_4(L)$.
\end{Corollary}

\begin{proof}
The quotient $L/(\gamma_{c+1}(L)+L'')$ is metabelian of class at most $c$.
Theorem~\ref{thm:C} and functoriality of the dimension series give the
first inclusion.  For $c=3$, the inclusion $L''\subseteq\gamma_4(L)$
gives the second assertion.
\end{proof}

\begin{Corollary}\label{cor:odd-metabelian}
\leanproof{\leanrepository/LieRings/Publication/Results.lean\#L217-L227}
Let $L$ be a metabelian Lie ring. Then
\[
                    D_{2c-1}(L)\subseteq\gamma_{c+1}(L),
                    \qquad c\geq2.
\]
\end{Corollary}

\begin{proof}
This follows from Corollary~\ref{cor:general-odd}, since $L''=0$.
\end{proof}

\Needspace{9\baselineskip}
\begin{Corollary}\label{cor:nilpotent-central-derived}
\leanproof{\leanrepository/LieRings/Publication/Results.lean\#L229-L237}
Let $c\geq2$ be an integer, and let $L$ be a nilpotent Lie ring of class at most $c$. Then
\[
                         D_{2c-1}(L)\subseteq L''\cap Z(L).
\]
\end{Corollary}

\begin{proof}
Corollary~\ref{cor:general-odd} gives
$D_{2c-1}(L)\subseteq L''$, while Theorem~\ref{thm:B} gives
\[
 [D_{2c-1}(L),L]=\gamma_{2c}(L)=0,
\]
so $D_{2c-1}(L)$ is central.
\end{proof}

\Needspace{10\baselineskip}
\begin{Corollary}\label{cor:omega}
\leanproof{\leanrepository/LieRings/Publication/Results.lean\#L239-L245}
Let $L$ be a metabelian Lie ring. Then \(D_\omega(L) = \gamma_\omega(L)\).
\end{Corollary}

\begin{proof}
By~\eqref{eq:lower-in-dimension} and
Corollary~\ref{cor:odd-metabelian},
\[
 \bigcap_{n\geq1}\gamma_n(L)
 \subseteq\bigcap_{n\geq1}D_n(L)
 \subseteq\bigcap_{c\geq2}\gamma_{c+1}(L)
 =\bigcap_{n\geq1}\gamma_n(L).\qedhere
\]
\end{proof}

\section{Formal verification}\label{sec:formalization}

\subsection{Verification of proofs in Lean}

Most results in this paper, including theorems~\ref{thm:A}, \ref{thm:B}
and~\ref{thm:C}, have been formally verified in Lean~4.  Each formally
verified statement bears a clickable ``Lean 4'' badge at its upper right,
linking directly to the corresponding proof in the source
code~\cite{IoninPetrovLean2026}.  The formalization contains
no \texttt{sorry} placeholders or \texttt{axiom} declarations.

The formalization includes the integral PBW theorem in the form used in
\S~\ref{sec:preliminaries}, including the injectivity of
$L\to\U(L)$.  It also includes the results
\[
 D_2(L)=\gamma_2(L),
 \qquad
 D_3(L)=\gamma_3(L),
 \qquad
 2D_4(L)\subseteq\gamma_4(L),
\]
and the equality of the two series for free Lie rings,
following~\cite{BartholdiPassi2015}.

\subsection{Verification of computations in GAP}
\label{subsec:gap-verification}

The following code was run with GAP~4.14.0 and NQ~2.5.11.  It constructs
the group $M$ from \S~\ref{subsec:group-plateau}, verifies
$|a|=2^{14}$ and $w=a^{2^{13}}\neq1$, and constructs the epimorphism
$\varphi\colon\gamma_4(M)\to C_{2^{14}}$ used to define $K$.
In class at most four, the second derived subgroup is generated by
commutators of basic commutators of weight two; these are the additional
relators in the listing.  All assertions in the listing pass.

\begin{lstlisting}[style=gapcode,
  caption={Verification of the group in \S~\ref{subsec:group-plateau}.},
  label={lst:group-plateau}]
LoadPackage("nq");;
SetInfoLevel(InfoNQ, 0);
SetAssertionLevel(1);
F := FreeGroup(6);;
x := GeneratorsOfGroup(F);;
d := List([1..6], i -> LeftNormedComm([x[4],x[6],x[i]]));;
rels := [
  x[1]^8*d[3]^4*d[2],
  x[2]^64*d[3]^-16*d[1]^-1*d[5]^16,
  x[3]^16384*d[2]^16*d[1]^-4,
  x[5]^1024*d[2]^16,
  d[5]^2048, d[1]^16, d[2]^128,
  Comm(d[1],x[1])*Comm(d[2],x[2])^-8,
  Comm(d[2],x[2])^-8*Comm(d[3],x[3])^2048
];;
# In class four these relators impose metabelianity.
cs := [];;
for i in [1..6] do
  for j in [1..i-1] do Add(cs, Comm(x[i],x[j])); od;
od;
for i in [1..Length(cs)] do
  for j in [1..i-1] do Add(rels, Comm(cs[i],cs[j])); od;
od;
# Retain the images of the six original generators.
data := NqCallANU_NQ(rec(group := F/rels, class := 4));;
coll := NqInitFromTheLeftCollector(data);;
M := NqPcpGroupByCollector(coll, data);;
y := List(data.Images, g -> NqPcpElementByWord(coll,g));;
a := LeftNormedComm([y[4],y[6],y[3],y[3]]);;
w := Comm(y[1],y[2]^256)*Comm(y[1],y[3]^16384)
     *Comm(y[2],y[3]^32768);;
top := LowerCentralSeries(M)[4];;
Assert(1, Order(a) = 2^14 and w = a^(2^13));
# The Smith basis defines the projection onto a cyclic factor.
basis := Pcp(top, "snf");;
orders := RelativeOrdersOfPcp(basis);;
alpha := ExponentsByPcp(basis, a);;
j := First([1..Length(basis)],
           i -> orders[i] = 2^14 and IsOddInt(alpha[i]));;
C := CyclicGroup(IsPcpGroup, 2^14);;
z := GeneratorsOfGroup(C)[1];;
images := List(basis, b -> One(C));;
images[j] := z^(PowerModInt(alpha[j], -1, 2^14));;
phi := GroupHomomorphismByImages(top, C,
         GeneratorsOfPcp(basis), images);;
Assert(1, Image(phi,a) = z and Image(phi,w) = z^(2^13));
Print("order(a) = ", Order(a), ", order(w) = ", Order(w), "\n");
Print("phi(a) = z, phi(w) = z^8192\n");
QUIT_GAP(0);
\end{lstlisting}

The output is
\begin{verbatim}
order(a) = 16384, order(w) = 2
phi(a) = z, phi(w) = z^8192
\end{verbatim}

\newcommand{\doilink}[1]{\href{https://doi.org/#1}{doi:#1}}

\end{document}